\documentclass{amsart}

\usepackage{amssymb}
\usepackage{amsfonts}
\usepackage{amsmath}
\usepackage{amsthm}
\usepackage{extarrows}
\usepackage{graphicx}
\usepackage{mathrsfs}
\usepackage{dsfont}
\usepackage{amscd}
\usepackage{multirow}
\usepackage[all]{xy}
\usepackage[scaled]{helvet} 
\usepackage{courier} 
\normalfont
\usepackage[T1]{fontenc}
\usepackage{calligra}
\usepackage{verbatim}
\usepackage[usenames,dvipsnames]{color}
\usepackage[colorlinks=true,linkcolor=Blue,citecolor=Violet]{hyperref}

\newcommand{\N}{{\mathds{N}}}
\newcommand{\Z}{{\mathds{Z}}}
\newcommand{\Q}{{\mathds{Q}}}
\newcommand{\R}{{\mathds{R}}}
\newcommand{\C}{{\mathds{C}}}
\newcommand{\T}{{\mathds{T}}}

\newcommand{\D}{{\mathfrak{D}}}
\newcommand{\A}{{\mathfrak{A}}}
\newcommand{\B}{{\mathfrak{B}}}

\newcommand{\G}{{\mathfrak{G}}}

\newcommand{\Qadic}[1]{{\mathds{Z}\left[\frac{1}{#1}\right]}}

\newcommand{\solenoid}[1]{\mathcal{S}_{\,#1}}
\newcommand{\ncsolenoid}[1]{{\alg{S}_{#1}}}
\newcommand{\bigslant}[2]{{\raisebox{.2em}{$#1$}\left/\raisebox{-.2em}{$#2$}\right.}}

\newcommand{\Lip}{{\mathsf{L}}}

\newcommand{\qpropinquity}[1]{{\mathsf{\Lambda}_{#1}}}
\newcommand{\propinquity}[1]{{\mathsf{\Lambda}^\ast_{#1}}}

\newcommand{\Kantorovich}[1]{{\mathsf{mk}_{#1}}}

\newcommand{\Haus}[1]{{\mathsf{Haus}_{#1}}}

\newcommand{\StateSpace}{{\mathscr{S}}}

\newcommand{\mongekant}{{Mon\-ge-Kan\-to\-ro\-vich metric}}

\newcommand{\Lqcms}{{\JLL} quan\-tum com\-pact me\-tric spa\-ce}

\newcommand{\unit}{1}

\newcommand{\sa}[1]{{\mathfrak{sa}\left({#1}\right)}}

\newcommand{\JLL}{Lei\-bniz}

\newcommand{\dom}[1]{{\operatorname*{dom}({#1})}}

\newcommand{\diam}[2]{{\mathrm{diam}\left({#1},{#2}\right)}}

\newcommand{\bridgereach}[2]{{\varrho\left({#1}\middle|{#2}\right)}}
\newcommand{\bridgeheight}[2]{{\varsigma\left({#1}\middle|{#2}\right)}}

\newcommand{\bridgelength}[2]{{\lambda\left({#1}\middle|{#2}\right)}}

\newcommand{\bridgenorm}[2]{{\mathsf{bn}_{ {#1}  }\left({#2}\right)}}

\newcommand{\alg}[1]{{\mathfrak{#1}}}

\theoremstyle{plain}
\newtheorem{theorem}{Theorem}[section]

\newtheorem{corollary}[theorem]{Corollary}

\newtheorem{lemma}[theorem]{Lemma}
\newtheorem{proposition}[theorem]{Proposition}

\newtheorem{theorem-definition}[theorem]{Theorem-Definition}

\theoremstyle{definition}
\newtheorem{definition}[theorem]{Definition}

\newtheorem{notation}[theorem]{Notation}

\newtheorem{hypothesis}[theorem]{Hypothesis}

\theoremstyle{remark}

\newtheorem{remark}[theorem]{Remark}

\renewcommand{\geq}{\geqslant}
\renewcommand{\leq}{\leqslant}

\numberwithin{equation}{section}

\begin{document}

\title{Noncommutative Solenoids and the Gromov-Hausdorff Propinquity}
\author{Fr\'{e}d\'{e}ric Latr\'{e}moli\`{e}re}
\thanks{This work is part of the project supported by the grant H2020-MSCA-RISE-2015-691246-QUANTUM DYNAMICS}
\email{frederic@math.du.edu}
\urladdr{http://www.math.du.edu/\symbol{126}frederic}
\address{Department of Mathematics \\ University of Denver \\ Denver CO 80208}

\author{Judith Packer}
\email{packer@euclid.colorado.edu}
\urladdr{http://spot.colorado.edu/\symbol{126}packer/}
\address{Department of Mathematics \\ University of Colorado \\ Boulder CO 80309}
\thanks{This work was partially supported by a grant from the Simons Foundation (\#316981 to Judith Packer)}
\date{\today}

\subjclass[2000]{Primary:  46L89, 46L30, 58B34.}
\keywords{Noncommutative metric geometry, Gromov-Hausdorff convergence, Monge-Kantorovich distance, Quantum Metric Spaces, Lip-norms}

\begin{abstract}
We prove that noncommutative solenoids are limits, in the sense of the Gromov-Hausdorff propinquity, of quantum tori. From this observation, we prove that noncommutative solenoids can be approximated by finite dimensional quantum compact metric spaces, and that they form a continuous family of quantum compact metric spaces over the space of multipliers of the solenoid, properly metrized.
\end{abstract}

\maketitle


\section{Introduction}

The quantum Gromov-Hausdorff propinquity, introduced by the first author \cite{Latremoliere13,Latremoliere13b}, is a distance on quantum compact metric spaces which extends the topology of the Gromov-Hausdorff distance \cite{Gromov,Gromov81}. Quantum metric spaces are generalizations of Lipschitz algebras \cite{Weaver99} first discussed by Connes \cite{Connes89} and formalized by Rieffel \cite{Rieffel98a}. The propinquity strengthens Rieffel's quantum Gromov-Hausdorff distance \cite{Rieffel00} to be well-adapted to the C*-algebraic framework, in particular by making *-isomorphism a necessary condition for distance zero \cite{Latremoliere15b}. The propinquity thus allows us to address questions from mathematical physics, such as the problem of finite dimensional approximations of quantum space times \cite{Connes97,Madore91,Douglas01,Taylor01},\cite[Ch. 7]{Madore}. Matricial approximations of physical theory motivates our project, which requires, at this early stage, the study of many different examples of quantum spaces. 

Recently, the first author proved that quantum tori form a continuous family for the propinquity, and admit finite dimensional approximations via so-called fuzzy tori \cite{Latremoliere13c}. This paper, together with the work on AF algebras done in \cite{Latremoliere15d}, explores the connection between our geometric approach to limits of C*-algebras and the now well studied approach via inductive limits, which itself played a role is quantum statistical mechanics \cite{Bratteli79}. We thus bring noncommutative solenoids, studied by the authors in \cite{Latremoliere11c,Latremoliere13d,Latremoliere14d}, and which are inductive limits of quantum tori, into the realm of noncommutative metric geometry. Our techniques apply to more general inductive limits on which projective limits of compact metrizable groups act ergodically. Noncommutative solenoids are interesting examples since they also are C*-crossed products, whose metric structures are still a challenge to understand. Irrational noncommutative solenoids \cite{Latremoliere11c} are non-type I C*-algebras, and many are even simple, thus they are examples of quantum spaces which are far from commutative. 

In our main result, we prove that noncommutative solenoids are limits, for the quantum Gromov-Hausdorff propinquity, of quantum tori. As corollaries, we then show that the map from the solenoid group to the family of noncommutative solenoids is continuous for the quantum propinquity, and that noncommutative solenoids are limits of fuzzy tori, namely C*-crossed products of finite cyclic groups acting on themselves by translation. As noncommutative solenoids have nontrivial $K_1$ group \cite{Latremoliere11c}, they are not AF algebras, so our proof that they are limits of finite dimensional C*-algebras illustrates the difference and potential usefulness of our metric geometric approach. Moreover, noncommutative solenoids' connection with wavelet theory \cite{Latremoliere14d} means that our result is a first step in what could be a metric approach to wavelet theory, by means of finite dimensional approximations. Last, metric approximations may prove a useful tool in the study of modules over noncommutative solenoids, initiated in \cite{Latremoliere13d, Latremoliere14d}, as recent research in noncommutative metric geometry is concerned in part with the category of modules over quantum metric spaces \cite{Rieffel15}

Noncommutative solenoids, introduced in \cite{Latremoliere11c} and studied further in \cite{Latremoliere13d,Latremoliere14d} by the authors, are the twisted group C*-algebras of the Cartesian square of the subgroups of $\Q$ consisting of the $p$-adic rationals for some $p\in \N\setminus\{0,1\}$. We begin with the classification of the multipliers of these groups.

\begin{theorem-definition}[\cite{Latremoliere11c}]\label{solenoid-def}
Let $p\in\N\setminus\{0,1\}$. The inductive limit of:
\begin{equation*}
\xymatrix{
\Z \ar^{k\mapsto pk}[r] & \Z \ar^{k\mapsto pk}[r] & \Z \ar^{k\mapsto pk}[r] &\cdots
}
\end{equation*}
is the group of $p$-adic rational numbers:
\begin{equation*}
\Qadic{p} = \left\{ \frac{q}{p^k} : q\in \Z, k \in \N \right\}\text{.}
\end{equation*}
The Pontryagin dual of $\Qadic{p}$ is the solenoid group:
\begin{equation*}
\begin{split}
\solenoid{p} &= \xymatrix{\varprojlim \T & \ar_{z\mapsto z^p}[l] \T & \ar_{z\mapsto z^p}[l] \T & \ar_{z\mapsto z^p}[l] \cdots} = \left\{ \left(z_n\right)_{n\in\N} \in \T^\N : \forall n\in \N \quad z_{n+1}^p = z_n \right\}\text{,}
\end{split}
\end{equation*}
where the dual pairing is given, for all $q\in\Z$, $k\in\N$, and $(z_n)_{n\in\N}\in\solenoid{p}$, by $\left< \frac{q}{p^k}, (z_n)_{n\in\N} \right> = z_k^q \text{.}$

For any $\theta = (\theta_n)_{n\in\N}\in\solenoid{p}$, and for all $q_1,q_2,q_3,q_4 \in \Z$ and $k_1,k_2,k_3,k_4 \in \N$, we define:
\begin{equation*}
\Psi_\theta: \left(\left(\frac{q_1}{p^{k_1}},\frac{q_2}{p^{k_2}}\right),\left(\frac{q_3}{p^{k_3}},\frac{q_4}{p^{k_4}}\right)\right) = \theta_{k_1 + k_4}^{q_1 q_4}\text{.}
\end{equation*}

For any multiplier $f$ of $\Qadic{p}\times\Qadic{p}$, there exists a unique $\theta \in \solenoid{p}$ such that $f$ is cohomologous to $\Psi_\theta$.
\end{theorem-definition}

Thus, formally, noncommutative solenoids are defined by:

\begin{definition}\label{ncsolenoid-def}
A \emph{noncommutative solenoid} $\ncsolenoid{\theta}$, for some $\theta\in\solenoid{p}$, is the twisted group C*-algebra $C^\ast\left(\Qadic{p}\times\Qadic{p},\Psi_\theta\right)$. 
\end{definition}

We compute the $K$-theory of noncommutative solenoids in \cite{Latremoliere11c} in terms of the multipliers of $\Qadic{p}\times\Qadic{p}$, identified with elements on the solenoid via Theorem-Definition (\ref{solenoid-def}); we then classify noncommutative solenoids up to their multiplier.

As the compact group $\solenoid{p}^2$ acts on $\ncsolenoid{\theta}$ for any $\theta \in \solenoid{p}$ via the dual action, any continuous length function on $\solenoid{p}^2$ induces a quantum metric structure on $\ncsolenoid{\theta}$, as described in \cite{Rieffel98a}. A quantum metric structure is given by a noncommutative analogue of the Lipschitz seminorm as follows:

\begin{notation}
If $\A$ is a C*-algebra with unit, then the norm on $\A$ is denoted by $\|\cdot\|_\A$, while the unit of $\A$ is denoted by $\unit_\A$. The state space of $\A$ is denoted by $\StateSpace(\A)$, and the subspace of self-adjoint elements in $\A$ is denoted by $\sa{\A}$.
\end{notation}

\begin{definition}[\cite{Rieffel98a,Rieffel99,Latremoliere13}]
A pair $(\A,\Lip)$ is a \emph{\Lqcms} when $\A$ is a unital C*-algebra and $\Lip$ is a seminorm defined on some dense Jordan-Lie subalgebra $\dom{\Lip}$ of the space of self-adjoint elements $\sa{\A}$ of $\A$, called a \emph{Lip-norm}, such that:
\begin{enumerate}
\item $\{a\in\sa{\A} : \Lip(a) = 0 \} = \R\unit_\A$,
\item $\max\left\{\Lip\left(\frac{ab+ba}{2}\right), \Lip\left(\frac{ab-ba}{2i}\right)\right\} \leq \|a\|_\A\Lip(b) + \|b\|_\A\Lip(a)\text{,}$
\item the {\mongekant} $\Kantorovich{\Lip}$ dual to $\Lip$ on $\StateSpace(\A)$ by setting, for all $\varphi, \psi \in \StateSpace(\A)$ by $\Kantorovich{\Lip}(\varphi,\psi) = \sup\left\{|\varphi(a) - \psi(a)| : a\in\dom{\Lip}, \Lip(a) \leq 1 \right\}$ induces the weak* topology on $\StateSpace(\A)$,
\item $\Lip$ is lower semi-continuous with respect to $\|\cdot\|_\A$.
\end{enumerate}
\end{definition}

Classical examples of Lip-norms are given by the Lipschitz seminorms on the C*-algebras of $\C$-valued continuous functions on compact metric spaces. An important source of noncommutative example is given by:

\begin{theorem-definition}[{\cite{Rieffel98a}}]\label{Rieffel-Lip-norm-thm}
Let $\alpha$ be a strongly continuous action by *-automorphisms of a compact group $G$ on a unital $C^\ast$-algebra $\A$ and let $\ell$ be a continuous length function on $G$. For all $a\in\sa{\A}$, we define:
\begin{equation*}
\Lip_{\alpha,\ell}(a) = \sup\left\{ \frac{\|a-\alpha^g(a)\|_\A}{\ell(g)} : g \in \G, \text{$g$ is not the unit of $G$} \right\}\text{.}
\end{equation*}
Then $\Lip_{\alpha,\ell}$ is a Lip-norm on $\A$ if and only if $\alpha$ is ergodic, i.e. $\{a\in\A : \forall g \in G\quad \alpha^g(a) = a \} = \C\unit_\A$. We note that $\Lip_{\alpha,\ell}$ is always lower semi-continuous.
\end{theorem-definition}

Theorem (\ref{Rieffel-Lip-norm-thm}) is thus, in particular, applicable to any dual action on the twisted group C*-algebra of some discrete Abelian group, such as noncommutative solenoids or quantum tori. 

This paper continues the study of the geometry of classes of quantum compact metric spaces under noncommutative analogues of the Gromov-Hausdorff distance, with the perspective that such a new geometric approach to the study of C*-algebras may prove useful in mathematical physics and C*-algebra theory. Our focus in this paper is a noncommutative analogue of the Gromov-Hausdorff distance devised by the first author \cite{Latremoliere13} as an answer to many early challenges in this program, and whose construction begins with a particular mean to relate two {\Lqcms s} via an object akin to a correspondence.

\begin{definition}
A \emph{bridge} from a unital C*-algebra $\A$ to a unital C*-algebra $\B$ is a quadruple $(\D,\omega,\pi_\A,\pi_\B)$ where:
\begin{enumerate}
\item $\D$ is a unital C*-algebra,
\item the element $\omega$, called the \emph{pivot} of the bridge, satisfies $\omega\in\D$ and $\StateSpace_1(\D|\omega)\not=\emptyset$, where:
\begin{equation*}
\StateSpace_1(\D|\omega)= \left\{ \varphi \in \StateSpace(\D) : \varphi((1-\omega^\ast\omega))=\varphi((1-\omega \omega^\ast)) = 0 \right\}
\end{equation*}
is called the \emph{$1$-level set of $\omega$},
\item $\pi_\A : \A\hookrightarrow \D$ and $\pi_\B : \B\hookrightarrow\D$ are unital *-monomorphisms.
\end{enumerate}
\end{definition}

There always exists a bridge between any two arbitrary {\Lqcms s} \cite{Latremoliere13}. The quantum propinquity is computed from a numerical quantity called the length of a bridge. We will denote the Hausdorff (pseudo)distance associated with a (pseudo)metric $\mathrm{d}$ by $\Haus{\mathrm{d}}$ \cite{Hausdorff}.

First introduced in \cite{Latremoliere13}, the length of a bridge is computed from two numbers, the height and the reach of a bridge. The height of a bridge assesses the error we make by replacing the state spaces of the {\Lqcms s} with the image of the $1$-level set of the pivot of the bridge, using the ambient {\mongekant}. 

\begin{definition}
Let $(\A,\Lip_\A)$ and $(\B,\Lip_\B)$ be two {\Lqcms s}. The \emph{height} $\bridgeheight{\gamma}{\Lip_\A,\Lip_\B}$ of a bridge $\gamma = (\D,\omega,\pi_\A,\pi_\B)$ from $\A$ to $\B$, and with respect to $\Lip_\A$ and $\Lip_\B$, is given by:
\begin{multline*}
\max\left\{ \Haus{\Kantorovich{\Lip_\A}}(\StateSpace(\A), \left\{ \varphi\circ\pi_\A : \varphi \in \StateSpace_1(\D|\omega)\right\} ), \right.\\ \left. \Haus{\Kantorovich{\Lip_\B}}(\StateSpace(\B), \left\{\varphi\in\pi_\B : \varphi\in\StateSpace_1(\D|\omega)\right\}) \right\}\text{.}
\end{multline*}
\end{definition}

The second quantity measures how far apart the images of the balls for the Lip-norms are in $\A\oplus\B$; to do so, they use a seminorm on $\A\oplus\B$ built using the bridge:
\begin{definition}[\cite{Latremoliere13}]
Let $\A$ and $\B$ be two unital C*-algebras. The \emph{bridge seminorm} $\bridgenorm{\gamma}{\cdot}$ of a bridge $\gamma = (\D,\omega,\pi_\A,\pi_\B)$ from $\A$ to $\B$ is the seminorm defined on $\A\oplus\B$ by $\bridgenorm{\gamma}{a,b} = \|\pi_\A(a)\omega - \omega\pi_\B(b)\|_\D$ for all $(a,b) \in \A\oplus\B$.
\end{definition}

We implicitly identify $\A$ with $\A\oplus\{0\}$ and $\B$ with $\{0\}\oplus\B$ in $\A\oplus\B$ in the next definition, for any two spaces $\A$ and $\B$.

\begin{definition}[\cite{Latremoliere13}]
Let $(\A,\Lip_\A)$ and $(\B,\Lip_\B)$ be two {\Lqcms s}. The \emph{reach} $\bridgereach{\gamma}{\Lip_\A,\Lip_\B}$ of a bridge $\gamma = (\D,\omega,\pi_\A,\pi_\B)$ from $\A$ to $\B$, and with respect to $\Lip_\A$ and $\Lip_\B$, is given by:
\begin{equation*}
\Haus{\bridgenorm{\gamma}{\cdot}}\left( \left\{a\in\sa{\A} : \Lip_\A(a)\leq 1\right\} , \left\{ b\in\sa{\B} : \Lip_\B(b) \leq 1 \right\}  \right) \text{.}
\end{equation*}
\end{definition}

We thus choose a natural synthetic quantity to summarize the information given by the height and the reach of a bridge:

\begin{definition}[\cite{Latremoliere13}]
Let $(\A,\Lip_\A)$ and $(\B,\Lip_\B)$ be two {\Lqcms s}. The \emph{length} $\bridgelength{\gamma}{\Lip_\A,\Lip_\B}$ of a bridge $\gamma = (\D,\omega,\pi_\A,\pi_\B)$ from $\A$ to $\B$, and with respect to $\Lip_\A$ and $\Lip_\B$, is given by $\max\left\{\bridgeheight{\gamma}{\Lip_\A,\Lip_\B}, \bridgereach{\gamma}{\Lip_\A,\Lip_\B}\right\}\text{.}$
\end{definition}

The quantum Gromov-Hausdorff propinquity is constructed from bridges, though the construction requires some care. We refer to \cite{Latremoliere13} for the construction, and summarize here the properties which we need in this paper.

\begin{theorem-definition}[\cite{Latremoliere13}]\label{def-thm}
Let $\mathcal{L}$ be the class of all {\Lqcms s}. There exists a class function $\qpropinquity{}$ from $\mathcal{L}\times\mathcal{L}$ to $[0,\infty) \subseteq \R$ such that:
\begin{enumerate}
\item for any $(\A,\Lip_\A), (\B,\Lip_\B) \in \mathcal{L}$ we have:
\begin{equation*}
0\leq \qpropinquity{}((\A,\Lip_\A),(\B,\Lip_\B)) \leq \max\left\{\diam{\StateSpace(\A)}{\Kantorovich{\Lip_\A}}, \diam{\StateSpace(\B)}{\Kantorovich{\Lip_\B}}\right\}\text{,}
\end{equation*}
\item for any $(\A,\Lip_\A), (\B,\Lip_\B) \in \mathcal{L}$ we have:
\begin{equation*}
\qpropinquity{}((\A,\Lip_\A),(\B,\Lip_\B)) = \qpropinquity{}((\B,\Lip_\B),(\A,\Lip_\A))\text{,}
\end{equation*}
\item for any $(\A,\Lip_\A), (\B,\Lip_\B), (\alg{C},\Lip_{\alg{C}}) \in \mathcal{L}$ we have:
\begin{equation*}
\qpropinquity{}((\A,\Lip_\A),(\alg{C},\Lip_{\alg{C}})) \leq \qpropinquity{}((\A,\Lip_\A),(\B,\Lip_\B)) + \qpropinquity{}((\B,\Lip_\B),(\alg{C},\Lip_{\alg{C}}))\text{,}
\end{equation*}
\item for all $(\A,\Lip_\A), (\B,\Lip_\B) \in \mathcal{L}$ and for any bridge $\gamma$ from $\A$ to $\B$, we have $\qpropinquity{}((\A,\Lip_\A), (\B,\Lip_\B)) \leq \bridgelength{\gamma}{\Lip_\A,\Lip_\B}\text{,}$
\item for any $(\A,\Lip_\A), (\B,\Lip_\B) \in \mathcal{L}$, we have $\qpropinquity{}((\A,\Lip_\A),(\B,\Lip_\B)) = 0$ if and only if $(\A,\Lip_\A)$ and $(\B,\Lip_\B)$ are isometrically isomorphic, i.e. if and only if there exists a *-isomorphism $\pi : \A \rightarrow\B$ with $\Lip_\B\circ\pi = \Lip_\A$, or equivalently there exists a *-isomorphism $\pi : \A \rightarrow\B$ whose dual map $\pi^\ast$ is an isometry from $(\StateSpace(\B),\Kantorovich{\Lip_\B})$ into $(\StateSpace(\A),\Kantorovich{\Lip_\A})$,

\item if $\Xi$ is a class function from $\mathcal{L}\times \mathcal{L}$ to $[0,\infty)$ which satisfies Properties (2), (3) and (4) above, then $\Xi((\A,\Lip_\A), (\B,\Lip_\B)) \leq \qpropinquity{}((\A,\Lip_\A),(\B,\Lip_\B))$ for all $(\A,\Lip_\A)$ and $(\B,\Lip_\B)$ in $\mathcal{L}$,
\item the topology induced by $\qpropinquity{}$ on the class of classical metric spaces agrees with the topology induced by the Gromov-Hausdorff distance.
\end{enumerate}
\end{theorem-definition}

The study of finite dimensional approximations of quantum compact metric spaces for the quantum propinquity is an important topic in noncommutative metric geometry, with results about the quantum tori \cite{Latremoliere05,Latremoliere13c}, spheres \cite{Rieffel10c,Rieffel15}, and AF algebras \cite{Latremoliere15d}. It is in general technically very difficult to construct natural approximations, while their existence is only known under certain certain quantum topological properties (pseudo-diagonality) \cite{Latremoliere15}. Moreover, quantum tori have been an important test case for our theory, with work on the continuity of the family of quantum tori \cite{Latremoliere13c}, and perturbations of metrics for curved quantum tori \cite{Latremoliere15c}. We refer to \cite{Latremoliere15b} for a survey of the theory of quantum compact metric spaces and the Gromov-Hausdorff propinquity. 

Last, we note that all our results are valid for the dual Gromov-Hausdorff propinquity \cite{Latremoliere13b,Latremoliere14} and therefore for Rieffel's quantum Gromov-Hausdorff distance \cite{Rieffel00}.

\section{Lip-norms from projective limits of compact groups}

The first step in obtaining our results about noncommutative solenoids consists in constructing a natural metric on the countable product $\prod_{n\in\N} G_n$ of a sequence $(G_n)_{n\in\N}$ of compact metrizable groups. Our metric is inspired by a standard construction of metrics on the Cantor set, and is motivated by the desire to have the sequence of subgroups $\left(\prod_{n>N}G_n\right)_{N\in\N}$ converge to the trivial group for the induced Hausdorff distance. This latter property will be the key to our computation of estimates on the propinquity later on. Our metrics are constructed from length functions. We recall that $\ell$ is a length function on a group $G$ with unit $e$ when:
\begin{enumerate}
\item for any $x\in G$, the length $\ell(x)$ is $0$ if and only if $x  = e$,
\item $\ell(x) = \ell\left(x^{-1}\right)$ for all $x\in G$,
\item $\ell(xy) \leq \ell(x) + \ell(y)$ for all $x,y \in G$.
\end{enumerate}

\begin{hypothesis}\label{group-hyp}
Let $(G_n)_{n\in\N}$ be a sequence of compact metrizable groups, and for each $n \in\N$ let $\ell_n$ be a continuous length function on $G_n$. Let $M \geq \diam{G_0}{\ell_0}$. Let:
\begin{equation*}
\mathds{G} = \prod_{n\in\N} G_n = \left\{ (g_n)_{n\in\N} : \forall n\in\N \quad g_n \in G_n \right\}\text{,}
\end{equation*}
endowed with the product topology. With the pointwise operations, $\mathds{G}$ is a compact group. We denote the unit of $\mathds{G}$ by $1$ and, by abuse of notation, we also denote the unit of $G_n$ by $1$ for all $n\in\N$.
\end{hypothesis}
 
\begin{definition}\label{solenoid-length-def}
Let Hypothesis (\ref{group-hyp}) be given. We define the length function $\ell_\infty$ on $\mathds{G}$ by setting, for any $g = (g_n)_{n\in\N}$ in $\mathds{G}$:
\begin{equation*}
\ell_\infty(g) = \inf\left\{ \varepsilon > 0 : \forall n \in\N\quad n < \frac{M}{\varepsilon} \implies \ell_n(g_n) \leq \varepsilon \right\}\text{.}
\end{equation*}
\end{definition}

The basic properties of our metric are given by:

\begin{proposition}\label{solenoid-length-prop}
Assume Hypothesis (\ref{group-hyp}). The length function $\ell_\infty$ on $\mathds{G} = \prod_{n\in\N} G_n$ from Definition (\ref{solenoid-length-def}) is continuous for the product topology on the compact group $\mathds{G}$, and thus metrizes this topology. Moreover, if for all $N \in \N$, we set 
$\mathds{G}^{(N)} = \left\{ (g_n)_{n\in\N} \in \mathds{G} : \forall j\in\{0,\ldots, N\} \quad g_j = 1 \right\}\text{,}$
then $\mathds{G}^{(N)}$ is a closed subgroup of $\mathds{G}$ and:
\begin{equation}\label{diam-solenoid-eq}
\diam{\mathds{G}^{(N)}}{\ell_\infty} \leq \frac{M}{N+1} \text{,}
\end{equation}
and thus in particular, if $1\in\mathds{G}$ is the unit of $\mathds{G}$:
\begin{equation}\label{cv-solenoid-eq}
\lim_{N\rightarrow\infty} \Haus{\ell_\infty}(\mathds{G}^{(N)}, \{1\}) = 0\text{.}
\end{equation}
\end{proposition}

\begin{proof}
We easily note that $\diam{\mathds{G}}{\ell_\infty} \leq \diam{G_0}{\ell_0}$. Indeed, if $g = (g_n)_{n\in\N} \in \mathds{G}$ then for $n = 0 < 1 = \frac{M}{\diam{G_0}{\ell_0}}$ we have $\ell_0(g_0) \leq \diam{G_0}{\ell_0}$. So by definition, $\ell_\infty(g) \leq \diam{G_0}{\ell_0}$.

Now, let $N\in\N$. We observe that if $g = (g_n)_{n\in\N} \in \mathds{G}^{(N)}$, then for all $n\leq N < \frac{M}{\frac{M}{N+1}}$ we have $\ell_n(g_n) = 0 \leq \frac{M}{N+1}$. Thus, $\ell_\infty(z) \leq \frac{M}{N+1}$. 

This proves both Expressions (\ref{diam-solenoid-eq}) and (\ref{cv-solenoid-eq}).

Assume now that $(g^m)_{m\in\N}$ converges in $\mathds{G}$ to some $g$, i.e. converges pointwise. Let $\varepsilon > 0$. Let $N = \lfloor \frac{M}{\varepsilon} \rfloor$. For each $j \in \{0,\ldots,N\}$, there exists $K_j \in \N$ such that for all $m \geq K_j$, we have $\ell_j(g^m_j g_j^{-1}) \leq \varepsilon$, by pointwise convergence. Let $K = \max\{K_j : j \in \{0,\ldots,N\}\}$. Then by construction, for all $m \geq K$, we have, for all $n < \frac{M}{\varepsilon}$, that $\ell_n (g^m_n g_n^{-1}) \leq \varepsilon$, so $\ell_\infty(g^m g^{-1})\leq \varepsilon$. Thus $\ell_\infty$ is continuous and induces a weaker topology on $\mathds{G}$ than the topology of pointwise convergence.

Assume now that $\ell_\infty((g_n)_{n\in\N}) = 0$. Fix $k\in\N$. Let $N > k$. Then $\ell_\infty((g_n)_{n\in\N})\leq\frac{M}{N+1}$. Thus by definition, $\ell_k(g_k) \leq \frac{M}{N+1}$ for all $N>k$. Thus $\ell_k(g_k) = 0$ for all $k\in\N$ and thus $g_k$ is the unit of $G_k$ for all $k\in\N$. 

Thus the topology induced by $\ell_\infty$ is Hausdorff, and thus, as the product topology on $\G^\N$ is compact by Tychonoff theorem, $\ell_\infty$ induces the product topology on $\G^\N$. This could also be easily verified directly.
\end{proof}

We shall apply Definition (\ref{solenoid-length-def}) and Proposition (\ref{solenoid-length-prop}) to projective limits, and thus we record the following corollary. \emph{We note that all our projective sequences of groups involve only epimorphisms.}

\begin{corollary}\label{solenoid-length-cor}
Let $\xymatrix{G_0 & \ar@{->>}[l]_{\rho_0} G_1 & \ar@{->>}[l]_{\rho_1} G_2 & \ar@{->>}[l]_{\rho_2} \cdots} = (G_n,\rho_n)_{n\in\N}$ be a projective sequence of compact metrizable groups, and let $\ell_n$ be a continuous length function on $G_n$ for all $n\in\N$. Let $M \geq \diam{G_0}{\ell_0}$. Let:
\begin{equation*}
G = \varprojlim (G_n,\rho_n)_{n\in\N} = \left\{ (g_n)_{n\in\N} \in \prod_{n\in\N} G_n : \forall n \in \N \quad g_n = \rho_{n}(g_{n+1}) \right\}\text{.}
\end{equation*}
The restriction to $G$ of the length function $\ell_\infty$ on $\prod_{n\in\N} G_n$ from Definition (\ref{solenoid-length-def}) metrizes the projective topology on $G$; moreover if $G_N = G \cap \mathds{G}^{(N)}$ for all $N\in\N$, then $\Haus{\ell_\infty}(G_N, \{1\}) \leq \frac{M}{N+1}$ with $1\in G$ the unit of $G$.
\end{corollary}

\begin{proof}
This is all straightforward as $G$ is a closed subgroup of $\mathds{G}$.
\end{proof}

We begin our study of quantum metrics on inductive limits with the observation that the proof of \cite[Theorem 3.83]{Latremoliere15b} includes the following fact, which will be of great use to us in view of Corollary (\ref{solenoid-length-cor}): 

\begin{lemma}\label{main-lemma}
Let $G$ be a compact metrizable group, $H\subseteq G$ be a normal closed subgroup, $\ell$ a continuous length function on $G$ and $\A$ a unital C*-algebra endowed with a strongly continuous ergodic action $\alpha$ of $G$. Let $K = \bigslant{G}{H}$ and let $\ell_K$ be the continuous length function $\ell_K : k \in K \mapsto \inf\{ \ell(g) : g \in kH \}$ where $kH$, for any $k\in K$, is the coset associated with $k$.

Let $\A_K = \{ a\in\A : \forall g \in H \quad \alpha^g(a) = a \}$ be the fixed point C*-subalgebra of $\A$ for the action $\alpha$ of $K$ on $\A$. Note that $\alpha$ induces an ergodic, strongly continuous action $\beta$ of $K$ on $\A_K$. Using Theorem (\ref{Rieffel-Lip-norm-thm}), Let $\Lip$ be the Lip-norm on $\A$ given by the action $\alpha$ of $G$ and the length function $\ell$, and let $\Lip_K$ be the Lip-norm on $\A_K$ given by the action $\beta$ of $K$ and the length function $\ell_K$. Then:
\begin{equation*}
\qpropinquity{}((\A,\Lip), (\A_K,\Lip_K)) \leq \diam{H}{\ell} \text{.}
\end{equation*}
\end{lemma}

\begin{proof}
We first note that since $H$ is closed, $\ell_K$ is easily checked to be a length function on $K$. Moreover, if $\pi : G\twoheadrightarrow K$ is the canonical surjection, then the trivial inequality $\ell_K(\pi(g)) \leq \ell(g)$ for all $g\in G$ proves that $\ell_K$ is continuous on $K$ since $g \in G \mapsto \ell_K(\pi(g))$ is $1$-Lipschitz, by characterization of continuity for the final topology on $K$.

Let $\mu$ be the Haar probability measure on $H$. For all $a\in\A$, we define:
\begin{equation*}
\mathds{E}(a) = \int_H \alpha^g(a) \,d\mu(g) \text{.}
\end{equation*}
A standard argument shows that $\mathds{E}$ is a unital conditional expectation on $\A$ with range $\A_K$. In particular, $\mathds{E}$ maps $\sa{\A}$ onto $\sa{\A_K}$.

Moreover, we note that since $H$ is normal, we have $gH = Hg$ for all $g\in 
G$, and thus:
\begin{equation*}
\begin{split}
\Lip_K(\mathds{E}(a)) &= \sup\left\{\frac{\left\|\alpha^g\left(\int_H \alpha^h(a)\,d\mu(h)\right) - \int_H \alpha^h(a)\,d\mu(h)\right\|_\A}{\ell(g)} : g\in G\setminus\{1\} \right\}\\
&= \sup\left\{\frac{\left\|\int_{gH} \alpha^{h}(a)\,d\mu(h) - \int_H \alpha^h(a)\,d\mu(h)\right\|_\A}{\ell(g)} : g\in G\setminus\{1\} \right\}\\
&= \sup\left\{\frac{\left\| \int_{Hg} \alpha^{h}(a)\,d\mu(h)-\int_H \alpha^{h}(a)\,d\mu(h)\right\|_\A}{\ell(g)} : g \in G\setminus\{1\} \right\}\\
&\leq \sup\left\{\frac{\int_{H} \left\|\alpha^{hg}(a) - \alpha^h(a) \,d\mu(h)\right\|_\A}{\ell(g)} : g \in G\setminus\{1\} \right\}\\
&= \sup\left\{\frac{\left\|a-\alpha^{g}(a)\right\|_\A}{\ell(g)} : g \in G\setminus\{1\} \right\} = \Lip(a) \text{.}
\end{split}
\end{equation*}
Hence $\mathds{E}$ is a weak contraction from $(\A,\Lip)$ onto $(\A_K,\Lip_K)$. 

Let now $\mathrm{id}$ be the identity operator on $\A$ and $\vartheta : \A_K \hookrightarrow\A$ be the canonical inclusion map. We thus define a bridge $\gamma = (\A,\unit_\A,\vartheta,\mathrm{id})$ from $\A_K$ to $\A$, whose height is null since its pivot is $\unit_\A$. We are thus left to compute the reach of $\gamma$.

To begin with, if $a\in\sa{\A_K}$ with $\Lip_K(a) \leq 1$, then an immediate computation proves that $\Lip(a) = \Lip_K(a) \leq 1$ and thus $\|a\unit_\A - \unit_\A a \|_\A = 0$.

Now let $a\in\sa{\A}$ with $\Lip(a) \leq 1$. Then $\Lip_K(\mathds{E}(a))\leq 1$, and we have:
\begin{equation*}
\begin{split}
\left\|a - \mathds{E}(a)\right\|_\A &= \left\|\int_H \alpha^h(a) - a \, d\mu(h) \right\|_\A \text{since $\mu$ probability measure,}\\
&\leq \int_H \|\alpha^h(a)-a\|_\A \, d\mu(h)\\
&\leq \int_H \ell(h) \Lip(a) \, d\mu(h) \\
&\leq \int_H \diam{H}{\ell} \,d\mu(h) = \diam{H}{\ell} \text{.}
\end{split}
\end{equation*}
Thus, the reach, and hence the length of $\gamma$ is no more than $\diam{H}{\ell}$, which, by Theorem-Definition (\ref{def-thm}), concludes our proof for our lemma.
\end{proof}

We are now in a position to prove one of the main results of this paper.

\begin{theorem}\label{main-thm}
Let $\xymatrix{G_0 & \ar@{->>}[l]_{\rho_0} G_1 & \ar@{->>}[l]_{\rho_1} G_2 & \ar@{->>}[l]_{\rho_2} \cdots} = (G_n,\rho_n)_{n\in\N}$ be a projective sequence of compact metrizable groups, and for each $n\in\N$, let $\ell_n$ be a continuous length function on $G_n$. Let $\A$ be a unital C*-algebra endowed with a strongly continuous action $\alpha$ of $G = \varprojlim (G_n,\rho_n)_{n\in\N}$. Let $\varrho_n : G \twoheadrightarrow G_n$ be the canonical surjection for all $n\in\N$.

We endow $G$ with the continuous length function $\ell_\infty$ from Definition (\ref{solenoid-length-def}) for some $M\geq \diam{G_0}{\ell_0}$.

For all $N \in \N$, let:
\begin{equation*}
G^{(N)} = \ker\varrho_N = \left\{ (g_n)_{n\in\N} \in G : \forall n \in \{0,\ldots,N-1\} \quad g_n = 1 \right\} \trianglelefteq G
\end{equation*}
and let $\A_N$ be the fixed point C*-subalgebra of $\alpha$ restricted to $G^{(N)}$. We denote by $\alpha_n$ the action of $G_n$ induced by $\alpha$ on $\A_n$ for all $n\in\N$.

Moreover, for all $n\in \N$ and $g \in G_n$ we set:
\begin{equation*}
\ell_\infty^n(g) = \inf\left\{ \ell_\infty(h) : \varrho_n(h) = g \right\}\text{.}
\end{equation*}

If, for some $n\in\N$, the action of $G_n$ induced by $\alpha$ on $\A_n$ is ergodic, then:
\begin{enumerate}
\item $\alpha$ is ergodic on $\A$ and $\alpha_n$ is ergodic on $\A_n$ for all $n\in\N$
\item If $\Lip$ is the Lip-norm induced by $\alpha$ and $\ell_\infty$ on $\A$ and $\Lip_n$ is the Lip-norm induced by $\alpha_n$ and $\ell_\infty^n$ on $\A_n$ using Theorem (\ref{Rieffel-Lip-norm-thm}), then for all $n\in\N$:
\begin{equation*}
\qpropinquity{}\left((\A,\Lip),(\A_n,\Lip_n)\right) \leq \frac{M}{n+1}
\end{equation*}
and thus:
$\lim_{n\rightarrow\infty} \qpropinquity{}\left((\A,\Lip),(\A_n,\Lip_n)\right) = 0\text{.}$
\end{enumerate}
\end{theorem}

\begin{proof}
For any given $n\in\N$, the group $G_n$ is isomorphic to $\bigslant{G}{G^{(n)}}$ and we are in the setting of Lemma (\ref{main-lemma}) --- in particular, $\ell_\infty^n$ is a continuous length function on $G_n$ and $\alpha_n$ is a well-defined action.
 
We note that by construction, for all $n\in\N$:
\begin{equation}\label{triv-eq}
\left\{a\in\A_n :\forall g \in G_n\quad \alpha_n^g(a) = a\right\} = \left\{a\in\A_n : \forall g \in G \quad \alpha^g(a) = a \right\}\text{.}
\end{equation}

Let us now assume that the action $\alpha_n$ is ergodic for some $n\in\N$. Let $a\in\A$ such that for all $g \in G$ we have $\alpha^g(a) = a$. Then $a\in\A_n$ in particular, since $a$ is invariant by the action of $\alpha$ restricted to $G^{(n)}$. Moreover, $a$ is invariant by the action $\alpha_n$ by Expression (\ref{triv-eq}) and thus $a\in\C\unit_\A$. Thus $\alpha$ is ergodic. This, in turn, proves that for all $n\in\N$, the action $\alpha_n$ is ergodic by Expression (\ref{triv-eq}).

Thus, $\Lip$ and $\Lip_n$ are now well-defined. By Lemma (\ref{main-lemma}) and Corollary (\ref{solenoid-length-cor}), we obtain:
\begin{equation*}
\qpropinquity{}((\A,\Lip),(\A_n,\Lip_n)) \leq \frac{M}{n+1} \text{.}
\end{equation*}
This concludes our proof.
\end{proof}

Theorem (\ref{main-thm}) involves an ergodic action of a projective limit of compact groups on a unital C*-algebra and one may wonder when such actions exist. The following theorem proves that one may obtain such actions on inductive limits, under reasonable compatibility conditions. Thus the next theorem provides us with a mean to construct Leibniz Lip-norms on inductive limits of certain {\Lqcms s}.

\begin{theorem}\label{ergodic-thm}
Let $\xymatrix{G_0 & \ar@{->>}[l]_{\rho_0} G_1 & \ar@{->>}[l]_{\rho_1} G_2 & \ar@{->>}[l]_{\rho_2} \cdots} = (G_n,\rho_n)_{n\in\N}$ be a projective sequence of compact groups. Let:
\begin{equation*}
G = \left\{ (g_n)_{n\in\N} \in \prod_{n\in\N} G_n : \forall n\in\N \quad \rho_n(g_{n+1}) = g_{n} \right\}\text{,}
\end{equation*} 
noting that $G = \varprojlim (G_n,\rho_n)_{n\in\N}$.

Let $\xymatrix{\A_0 \ar@{^{(}->}[r]^{\varphi_0} & \A_1 \ar@{^{(}->}[r]^{\varphi_1} & \A_2 \ar@{^{(}->}[r]^{\varphi_2} & \cdots} = (\A_n,\varphi_n)_{n\in\N}$ be an inductive sequence of unital C*-algebras where, for all $n\in\N$, we assume:
\begin{enumerate}
\item $\varphi_n$ is a *-monomorphism,
\item there exists an ergodic action $\alpha_n$ of $G_n$ on $\A_n$,
\item for all $g = (g_n)_{n\in\N} \in G$ we have:
\begin{equation}\label{compatibility-eq}
\varphi_{n}\circ\alpha_n^{g_n} = \alpha_{n+1}^{g_{n+1}}\circ\varphi_n\text{.}
\end{equation}
\end{enumerate}
We denote by $\A$ the inductive limit of $(\A_n,\varphi_n)_{n\in\N}$.

Then there exists an ergodic strongly continuous action $\alpha$ of $G = \varprojlim(G_n,\rho_n)_{n\in\N}$ on $\A$. 
\end{theorem}

\begin{proof}
For all $(a_n)_{n\in\N} \in \prod_{n\in\N}\A_n$, we set $\|(a_n)\|_\infty = \limsup_{n\rightarrow\infty} \|a_n\|_{\A_n}$, which defined a C*-seminorm on $\prod_{n\in\N}\A_n$. The quotient of $\prod_{n\in\N}\A_n$ by $\{a\in\prod_{n\in\N}\A_n:\|a\|_\infty = 0\}$, endowed with the quotient seminorm of $\|\cdot\|_\infty$, which we still denote by $\|\cdot\|_\infty$, is a C*-algebra, which we denote by $\limsup_{n\rightarrow\infty}\A_n$. Let $\pi$ be the canonical surjection from $\prod_{n\in\N}\A_n$ onto $\limsup_{n\rightarrow\infty}\A_n$. 

Up to a *-isomorphism, $\A = \varinjlim(\A_n,\varphi_n)$ is the completion of the image by $\pi$ of the set:
\begin{equation*}
\A_\infty = \left\{ (a_n)_{n\in\N} : \exists N\in\N \quad \forall n > N \quad a_n = \varphi_{n-1}\circ\ldots\circ\varphi_N(a_N) \right\}\text{,}
\end{equation*}
in $\limsup_{n\rightarrow\infty} \A_n$. 

We begin with a useful observation. Let $a = (a_n)_{n\in\N}$ and $b = (b_n)_{n\in\N}$ in $\A_\infty$ with $\|a-b\|_\infty = 0$. Let $N\in \N$ such that, for all $n\geq N$, we have $a_{n+1} = \varphi_n(a_n)$ and $b_{n+1} = \varphi_n(b_n)$: note that by definition, such a number $N$ exists. If $\|a_N - b_N\|_\A > \varepsilon$ for some $\varepsilon > 0$, then since $\varphi_n$ is a *-monomorphism for all $n\in\N$, it is an isometry, and thus $\limsup_{n\rightarrow\infty} \|a_n-b_n\|_{\A_n} \geq \varepsilon$, which is a contradiction. Hence, for all $n\geq N$ we have $\|a_n - b_n\|_{\A_n} = 0$. Informally, if two sequences in $\A_\infty$ describe the same element of $\A$, then their predictable tails are in fact equal.

We now define the action of $G$ on $\A$. For $g = (g_n)_{n\in\N} \in G$ and $(a_n)_{n\in\N} \in \A_\infty$, we set $\alpha^g((a_n)_{n\in\N}) = \left(\alpha^{g_n}(a_n)\right)_{n\in\N}$, which is a *-morphism of norm $1$. Condition (\ref{compatibility-eq}) ensures that $\alpha^g$ maps $\A_\infty$ to itself. It induces an action of $G$ on $\pi(\A)$ by norm $1$ $\ast$-automorphisms in the obvious manner, and thus extends to $\A$ by continuity (we use the same notation for this extension). It is easy to check that $\alpha$ is an action of $G$ on $\A$.

Let $a \in \pi(\A_\infty)$ such that $\alpha^g(a) = a$ for all $g\in G$. Let $(a_n)_{n\in\N}\in\A_\infty$ with $\pi((a_n)_{n\in\N}) = a$. Let $N\in\N$ such that for all $n\geq N$, we have $a_{n+1} = \varphi_n(a_n)$. By definition of the action $\alpha$, we have for all $g = (g_n)_{n\in\N} \in G$ that $\alpha^g(a) = (\alpha_n^{g_n}(a_n))_{n\in\N}$, and we note that:
\begin{equation*}
\alpha_{n+1}^{g_{n+1}}\left(a_{n+1}\right) = \alpha_{n+1}^{g_{n+1}}\left(\varphi_n(a_n) \right) = \varphi_n\left(\alpha_n^{g_n}(a)\right) \text{,}
\end{equation*}
by Condition (\ref{compatibility-eq}). Thus by our earlier observation, we conclude that $\alpha_N^{g_N}(a_N) = a_N$ for all $g\in G$. Thus, as $\rho_N$ is surjective, and $\alpha_N$ is ergodic, we conclude that $a_N = \lambda \unit_{\A_N}$. Thus for all $n\geq N$ we have $a_n = \varphi_{n-1}\circ\cdots\circ\varphi_N(\lambda\unit_{\A_N})$. Consequently, $a \in \C\unit_\A$ by definition.

Now, let $\mu$ be the Haar probability measure on $G$ and define $\mathds{E}(a) = \int_G \alpha^g(a)\,d\mu(g)$ for all $a\in\A$. It is straightforward to check that $\mathds{E}(a)$ is invariant by $\alpha$ for all $a\in\A$. 

Let $a\in\A$ such that $\alpha^g(a) = a$ for all $g\in G$. Thus $\mathds{E}(a) = a$. Let $\varepsilon > 0$. There exists $a_\varepsilon \in \A_\infty$ such that $\|a-a_\varepsilon\|_\A \leq \frac{\varepsilon}{2}$. Now:
\begin{equation*}
\|\mathds{E}(a) - \mathds{E}(a_\varepsilon)\|_\A = \|\mathds{E}(a-a_\varepsilon)\|_\A \leq \|a-a_\varepsilon\|_\A\leq\frac{\varepsilon}{2}\text{.}
\end{equation*}
and yet $\mathds{E}(a_\varepsilon) \in \C\unit_\A$ since $G$ is ergodic on $\A_\infty$. Thus, as $\varepsilon > 0$ is arbitrary, $\mathds{E}(a)$ lies in the closure of $\C\unit_\A$, i.e. in $\C\unit_\A$, and thus $\alpha$ is ergodic.

Finally, again let $a\in\A$ and $\varepsilon > 0$, and let $a_\varepsilon \in \pi(\A_\infty)$ such that $\|a-a_\varepsilon\|_\A \leq \frac{\varepsilon}{3}$. Let $(a_n)_{n\in\N} \in \A_\infty$ such that $\pi((a_n)_{n\in\N}) = a_\varepsilon$. There exists $N\in\N$ such that $\varphi_{n}(a_n) = a_{n+1}$ for all $n\geq N$. Since $\alpha_N$ is strongly continuous, there exists a neighborhood $V$ of $1\in G_N$ such that $\|\alpha_N^g(a_N) - a_N\|_{\A_N} < \frac{\varepsilon}{3}$ for all $g \in V$. Let $W = \rho_N^{-1}(V)$ which is an open neighborhood of $1\in G$. Then since $\varphi_n$ is an isometry for all $n\in\N$, we have for all $g = (g_n)_{\in\N} \in W$:
\begin{equation*}
\|\alpha_n^{g_n}(a_n) - a_n\|_{\A_n} = \|\alpha_N^{g_N}(a_N) - a_N\|_{\A_N} \leq \frac{\varepsilon}{3}\text{.}
\end{equation*}
Thus for all $g\in W$ we have:
\begin{equation*}
\|a-\alpha^g(a)\|_\A \leq \|a-a_\varepsilon\|_\A + \|a_\varepsilon - \alpha^g(a_\varepsilon)\|_\A + \|\alpha^g(a_\varepsilon - a)\| \leq \varepsilon\text{.}
\end{equation*}
Thus $\alpha$ is strongly continuous.
\end{proof}

Thus, Theorem (\ref{ergodic-thm}) can provide ergodic, strongly continuous actions on certain inductive limits, which then fit Theorem (\ref{main-thm}) and provide us with convergence of certain {\Lqcms s} to inductive limit C*-algebras:
\begin{corollary}
We assume the same assumptions as Theorem (\ref{ergodic-thm}). Moreover, for each $n\in\N$, let $\ell_n$ be a continuous length function on $G_n$. Let $\ell_\infty$ and, for all $n\in\N$, let $\ell_\infty^n$ be given as in Theorem (\ref{main-thm}), for some $M \geq \diam{G_0}{\ell_0}$.

We denote by $\A$ the inductive limit of $(\A_n,\varphi_n)_{n\in\N}$.

Let $\alpha$ be the action of $G$ on $\A$ constructed in Theorem (\ref{ergodic-thm}). For all $n\in\N$, let $\B_n$ is the fixed point C*-subalgebra of the restriction of $\alpha$ to $\ker \rho_n$, let $\Lip_n$ be the Lip-norm defined from the restriction of $\alpha$ to $G_n$ on $\B_n$ using the length function $\ell_\infty^n$. If $\Lip$ is the Lip-norm on $\A$ induced by $\alpha$ and $\ell_\infty$ via Theorem (\ref{Rieffel-Lip-norm-thm}) then:
\begin{equation*}
\lim_{n\rightarrow\infty} \qpropinquity{}((\A,\Lip),(\B_n,\Lip_n)) = 0\text{.}
\end{equation*}
\end{corollary}

\begin{proof}
Apply Theorem (\ref{main-thm}) to Theorem (\ref{ergodic-thm}).
\end{proof}

\section{Approximation of noncommutative solenoids by quantum tori}

We apply the work of our previous section to the noncommutative solenoids. We begin by setting our framework. We begin with some notation.

\begin{notation}
For any $\theta \in \solenoid{p}$, the noncommutative solenoid $\ncsolenoid{\theta}$ is, by Definition (\ref{ncsolenoid-def}), the universal C*-algebra generated by unitaries $W_{x,y}$ with $x,y \in \Qadic{p}\times\Qadic{p}$, subject to the relations: $W_{x,y}W_{x',y'} = \Psi_\theta((x,y),(x',y'))W_{x+x',y+y'}$.
\end{notation}

By functoriality of the twisted group C*-algebra construction, we note that noncommutative solenoids are inductive limits of quantum tori. All the quantum tori in this paper are rotation C*-algebras, and we shall employ a slightly unusual notation, which will make our presentation clearer:

\begin{notation}
The rotation C*-algebra $\A_\theta$, for $\theta\in \T$, is the C*-algebra generated by two unitaries $U_\theta$ and $V_\theta$ which is universal for the relation $VU = \theta UV$.
\end{notation}

\begin{theorem}[\cite{Latremoliere11c}]\label{inductive-thm}
Let $p\in\N\setminus\{0\}$ and $\theta \in \solenoid{p}$. For each $n\in\N$, we define the map $\Theta_n : \A_{\theta_{2n}}\rightarrow\A_{\theta_{2n+2}}$ as the unique *-monomorphism such that:
\begin{equation*}
\Theta_n(U_{\theta_{2n}}) = U_{\theta_{2n+2}}^p \text{ and }\Theta_n(V_{\theta_{2n}}) = V_{\theta_{2n+2}}^p \text{.}
\end{equation*}
Then:
\begin{equation*}
\ncsolenoid{\theta} = \varinjlim (\A_{\theta_{2n}},\Theta_n)_{n\in\N}\text{.}
\end{equation*}
Moreover, the canonical injection $\rho_n$ from $\A_{\theta_{2n}}$ into $\ncsolenoid{\theta}$ is given by extending the map:
\begin{equation*}
U_{\theta_{2n}} \mapsto W_{\frac{1}{p^{n}}, 0} \text{ and }V_{\theta_{2n}} \mapsto W_{0,\frac{1}{p^{n}}} \text{.}
\end{equation*}
\end{theorem}

\begin{remark}
In Theorem (\ref{inductive-thm}), only the entries with even indices in the solenoid element defining the twist of the noncommutative solenoid are involved, since by our choice of multiplier in Theorem-Definition (\ref{solenoid-def}), the commutation relations between the canonical generators $W_{0,p^{-k}}$ and $W_{p^{-k},0}$ only involves these indices. Note however that the definition of the solenoid group implies that given all the even indices entries of one of its element, the entire group element is uniquely determined.
\end{remark}

We note that the dual action of $\solenoid{p}$ on any noncommutative solenoid may be obtained using Theorem (\ref{ergodic-thm}) and the dual actions on quantum tori.

We now have all our ingredients to prove the main result of this paper.

\begin{theorem}\label{ncsolenoid-qt-thm}
Let $\theta \in \solenoid{p}$ and $\ell$ a continuous length function on $\T^2$. We let $\ell_\infty$ be the length function of Definition (\ref{solenoid-length-def}) on $\solenoid{p}^2$ for $M = \diam{\T^2}{\ell}$. For all $n\in\N$ and all $z\in \T^2$, let:
\begin{equation*}
\ell_\infty^n (z) = \inf\left\{ \ell_\infty(\omega) : \omega \in \solenoid{p}^2, \omega = (z^{p^n}, z^{p^{n-1}}, \ldots, z, \ldots) \right\}\text{.}
\end{equation*}
Then $\ell_\infty^n$ is a continuous length function on $\T^2$. Let $\Lip_n$ be the Lip-norm on the quantum torus $\A_{\theta_{2n}}$ defined by $\ell_\infty^n$, the dual action of $\T^2$ on $\A_{\theta_{2n}}$, and Theorem-Definition (\ref{Rieffel-Lip-norm-thm}).

Let $\Lip$ be the Lip-norm on $\ncsolenoid{\theta}$ defined by the dual action $\alpha$ of $\solenoid{p}^2$ and the length $\ell_\infty$ via Theorem-Definition (\ref{Rieffel-Lip-norm-thm}). 

We then have, for all $n\in\N$:

\begin{equation*}
\propinquity{}((\ncsolenoid{\theta},\Lip),(\A_{\theta_{2n}},\Lip_n)) \leq \frac{\diam{\T^2}{\ell}}{n+1} \text{.}
\end{equation*}

In particular:
\begin{equation*}
\lim_{n\rightarrow\infty} \propinquity{}\left(\left(\ncsolenoid{\theta},\Lip\right), \left(\A_{\theta_{2n}},\Lip_n \right)\right) = 0\text{.}
\end{equation*}
\end{theorem}

\begin{proof}
Let $N\in\N$ and let $\solenoid{p,N} = \left\{ (z_n)_{n\in\N} : \forall n \leq N \quad z_n = 1  \right\}\text{.}$ If $\mathds{G} = \solenoid{p}^2$ then $\solenoid{p,N}^2 = \mathds{G}^{(N)}$ using the notation of Theorem (\ref{main-thm}). 

The quotient $\bigslant{\solenoid{p}}{\solenoid{p,N}}$ is given by:
\begin{equation*}
\left\{ (z_n)_{0\leq n \leq N} \in \T^{N+1} : \forall n \in \{0,\ldots,N\} \quad z_{n+1}^p = z_n \right\} \text{.}
\end{equation*}
The map $z\in \T \mapsto (z^{p^N}, z^{p^{N-1}}, \ldots,z)$ is an isomorphism from $\T$ onto $\bigslant{\solenoid{p}}{\solenoid{p,N}}$. Moreover, the dual of $\bigslant{\solenoid{p}}{\solenoid{p,N}}$ is isomorphic to the subgroup:
\begin{equation*}
Z_N = \left\{ \frac{q}{p^k} : k \in \{0,\ldots,N\} \right\}
\end{equation*}
of $\Qadic{p}$; this subgroup is trivially isomorphic to $\Z$ via the map $z\in \Z \mapsto \frac{z}{p^N}$. In fact, this isomorphism is also (up to changing the codomain to make it a monomorphism) the canonical injection of the $N^{\mathrm{th}}$ copy of $\Z$ to $\Qadic{p}$, with range $Z_N \triangleleft \Qadic{p}$, when writing $\Qadic{p}$ as the inductive limit of $\xymatrix{\Z \ar^{k\mapsto pk}[r] & \Z \ar^{k\mapsto pk}[r] &\cdots}$.

By Theorem (\ref{main-thm}), it is thus sufficient, to conclude, that we identify the fixed point C*-subalgebra of $\ncsolenoid{\theta}$ for the subgroup $\solenoid{p,N}^2$.

Let $\mu$ be the Haar probability measure on $\solenoid{p}^2$. As in the proof of Lemma (\ref{main-lemma}), We define the conditional expectation $\mathds{E}_N$ of $\ncsolenoid{\theta}$ by setting for all $a\in \ncsolenoid{\theta}$:
\begin{equation*}
\mathds{E}_N(a) = \int_{\solenoid{p,N}^2}\alpha^g(a)\,d\mu(g) \text{.}
\end{equation*}

Let $(z,y) \in \solenoid{p}^2$, and $q_1,q_2 \in \Z$, $k_1,k_2 \in \N$. By Theorem-Definition (\ref{solenoid-def}) and by definition of the dual action $\alpha$ of $\solenoid{p}^2$ on $\ncsolenoid{\theta}$, we compute:
\begin{equation*}
\begin{split}
\alpha^{z,y}\left(W_{\frac{q_1}{p^{k_1}},\frac{q_2}{p^{k_2}}}\right) &= z_{k_1}^{q_1} y_{k_2}^{q_2} W_{\frac{q_1}{p^{k_1}},\frac{q_2}{p^{k_2}}}\text{.}
\end{split}
\end{equation*}

Thus, if $(z,y) \in \solenoid{p,N}^2$ then $\alpha^{z,y}(W_{\frac{q_1}{p^{k_1}},\frac{q_2}{p^{k_2}}}) = W_{\frac{q_1}{p^{k_1}},\frac{q_2}{p^{k_2}}}$ for all $\frac{q_1}{p^{k_1}},\frac{q_2}{p^{k_2}} \in Z_N$. On the other hand, $\mathds{E}_N(W_{\frac{q_1}{p^{k_1}},\frac{q_2}{p^{k_2}}}) = 0$ for all $\frac{q_1}{p^{k_1}},\frac{q_2}{p^{k_2}} \not\in Z_N$.

Thus the range of $\mathds{E}_N$, which is the fixed point C*-subalgebra for $\solenoid{p,N}^2$, is the C*-subalgebra of $\ncsolenoid{\theta}$ generated by:
\begin{equation*}
\left\{ W_{\frac{1}{p^{k_1}},\frac{1}{p^{k_2}}} : \frac{q_1}{p^{k_1}},\frac{q_2}{p^{k_2}}\in Z_N \right\}\text{.}
\end{equation*}
Now, by definition:
\begin{equation*}
W_{\frac{1}{p^{k_1}},\frac{1}{p^{k_2}}} = \overline{\Psi_{\theta}\left(\left(\frac{q_1}{p^{k_1}},0\right),\left(0,\frac{q_2}{p^{k_2}}\right)\right)} \left(W_{\frac{1}{p^N},0}\right)^{q_1 k_1 p} \left(W_{0,\frac{1}{p^N}}\right)^{q_2 k_2 p}\text{.}
\end{equation*}

Thus, the range of $\mathds{E}_N$ is the C*-subalgebra of $\ncsolenoid{\theta}$ generated by $W_{\frac{1}{p^N},0}, W_{0,\frac{1}{p^N}}$. By Theorem (\ref{inductive-thm}), the range of $\mathds{E}_N$ is the image of $\A_{\theta_{2n}}$ in $\ncsolenoid{\theta}$ via the canonical injection $\rho_N$ defined in Theorem (\ref{inductive-thm}). Now, note that $\rho_N$ is an isometry from $\Lip_N$ to the Lip-norm $\Lip_{\solenoid{p,N}^2}$ defined by Theorem (\ref{Rieffel-Lip-norm-thm}), the restriction of the dual action $\alpha$ to $\solenoid{p,N}^2$, acting on $\mathds{E}_N(\ncsolenoid{\theta})$ (as in Lemma (\ref{main-lemma})). Thus:
\begin{equation*}
\qpropinquity{}((\mathds{E}_N(\ncsolenoid{\theta}), \Lip_{\solenoid{p,N}^2}), (\A_{\theta_{2n}},\Lip_N)) = 0 \text{.}
\end{equation*}

By Theorem (\ref{main-thm}), we thus conclude:
\begin{equation*}
\begin{split}
\qpropinquity{}((\ncsolenoid{\theta},\Lip), (\A_{\theta_{2N}},\Lip_N)) &= \qpropinquity{}((\ncsolenoid{\theta},\Lip), (\mathds{E}_N(\ncsolenoid{\theta}), \Lip_{\solenoid{p,N}^2}))\\
&\leq \diam{\ncsolenoid{p,N}^2}{\ell_\infty}\\
&\leq \frac{\diam{\T^2}{\ell}}{N+1}\text{ by Corollary (\ref{solenoid-length-cor}).}
\end{split}
\end{equation*}
This completes our proof.
\end{proof}

We note that since convergence for the quantum propinquity implies convergence in the sense of the Gromov-Hausdorff distance for classical metric spaces, we have proven that $(\T^2,\ell_\infty^n)_{n\in\N}$ converges to $\solenoid{p}^2$ in the Gromov-Hausdorff distance, using the notations of Theorem (\ref{ncsolenoid-qt-thm}). 

We begin with the immediate observation that, since quantum tori are limits of fuzzy tori for the quantum propinquity, so are the noncommutative solenoids.

\begin{corollary}
Let $p\in\N\setminus\{0\}$ and $\theta\in\solenoid{p}$. Fix a continuous length function $\ell$ on $\T^2$ and let $\ell_\infty$ be the induced length function on $\solenoid{p}^2$ given in Definition (\ref{solenoid-length-def}). 

There exists a sequence $(\omega_n)_{n\in\N} \in \T^\N$ and a sequence $(k_n)_{n\in\N}$ in $\N^\N$ with $\lim_{n\rightarrow\infty}k_n = \infty$, $\lim_{n\rightarrow\infty} |\theta_{2n} - \omega_n| = 0$, and $\omega_n^{k_n} = 1$ for all $n\in\N$, such that:
\begin{equation*}
\lim_{n\rightarrow\infty} \qpropinquity{}((C^\ast(\Z_{k_n}^2, \sigma_n), \Lip_n), (\ncsolenoid{\theta},\Lip)) = 0
\end{equation*}
where $\Z_k = \bigslant{\Z}{k\Z}$, $\Lip_n$ and $\Lip$ are the Lip-norms given by Theorem (\ref{Rieffel-Lip-norm-thm}) for the dual actions, respectively, of the groups of $k_n$ roots of unit and the solenoid group $\solenoid{p}$, and:
\begin{equation*}
\sigma_n : ((z_1,z_2),(y_1,y_2)) \in \Z_{k_n}^2\times \Z_{k_n}^2 \mapsto \exp(2i\pi\omega_n (z_1 y_2 - z_2 y_1)) \text{.} 
\end{equation*}
\end{corollary}

\begin{proof}
This follows from a standard diagonal argument using Theorem (\ref{ncsolenoid-qt-thm}) and \cite[Theorem 5.2.5]{Latremoliere13c}.
\end{proof}

Quantum tori form a continuous family for the quantum propinquity, and together with Theorem (\ref{ncsolenoid-qt-thm}), we thus can prove:
\begin{theorem}
Let $\ell$ be a continuous length function on $\T^2$. For each $\theta \in \solenoid{p}$, let $\Lip_\theta$ be the Lip-norm defined by Theorem (\ref{Rieffel-Lip-norm-thm}) for the dual action of $\solenoid{p}^2$ on $\ncsolenoid{\theta}$ and the continuous length function $\ell_\infty$ of Definition (\ref{solenoid-length-def}). 

The function $\theta \in \solenoid{p} \longmapsto \left(\ncsolenoid{\theta},\Lip_\theta\right)$ is continuous from $\solenoid{p}$ to the class of {\Lqcms s} endowed with the quantum Gromov-Hausdorff propinquity.
\end{theorem}

\begin{proof}
Fix some continuous length function $\mathsf{m}$ on $\T$. This length function need not be related to $\ell$. Its purpose is simply to provide us with a metric $\lambda_{\mathsf{m}}$ for the topology of $\solenoid{p}$.

Let $\varepsilon > 0$. Let $N\in \N$ be chosen so that $\frac{\diam{\T^2}{\ell}}{N+1} \leq \frac{\varepsilon}{3}$. By Theorem (\ref{ncsolenoid-qt-thm}), for all $\theta \in \solenoid{p}$, we have:
\begin{equation*}
\propinquity{}((\ncsolenoid{\theta},\Lip), (\A_{\theta_{2N}},\Lip)) \leq \frac{\varepsilon}{3}\text{.}
\end{equation*}

By \cite[Theorem 5.2.5]{Latremoliere13c}, there exists $\delta > 0$ such that, for all $\omega,\eta \in [0,1)$ with $\mathsf{m}(\omega \eta^{-1})\leq\delta$, we have $\propinquity{}((\A_\omega,\Lip),(\A_\eta,\Lip)) \leq \frac{\varepsilon}{3}$. 

Let $\varsigma = \min \left\{ \delta, \frac{\diam{\T}{\mathsf{m}}}{N+1} \right\}$. Let $\theta, \xi \in \solenoid{p}$ with $\lambda_{\mathrm{m}}(\theta,\xi) \leq \varsigma$. By definition of $\lambda_{\mathrm{m}}$, we have $\mathrm{m}(\theta_{2n} \xi_{2n}^{-1}) \leq \delta$. Consequently:
\begin{multline*}
\qpropinquity{}((\ncsolenoid{\theta},\Lip_\theta), (\ncsolenoid{\xi},\Lip_\xi)) \leq \qpropinquity{}((\ncsolenoid{\theta},\Lip_\theta),(\A_{\theta_{2N}},\Lip)) \\ + \qpropinquity{}((\A_{\theta_{2N}},\Lip),(\A_{\xi_{2N}},\Lip)) + \qpropinquity{}((\A_{\xi_{2N}},\Lip),(\ncsolenoid{\xi},\Lip_\xi)) \leq \varepsilon\text{,} 
\end{multline*}
which concludes our theorem.
\end{proof}

\providecommand{\bysame}{\leavevmode\hbox to3em{\hrulefill}\thinspace}
\providecommand{\MR}{\relax\ifhmode\unskip\space\fi MR }
\providecommand{\MRhref}[2]{%
  \href{http://www.ams.org/mathscinet-getitem?mr=#1}{#2}
}
\providecommand{\href}[2]{#2}

\vfill

\end{document}